\documentclass[a4paper,10pt,twoside]{article}

\usepackage[body={170mm,235mm}]{geometry}
\usepackage{a1}
\usepackage[english]{babel}
\usepackage[latin1]{inputenc} %permite o uso de acentos
\usepackage{amsfonts,amssymb}
\usepackage{amsmath}
\usepackage{graphicx}	

\usepackage{makeidx}
\makeindex

\def\mapright#1#2#3{\smash{\mathop{\hbox to
#3{\rightarrowfill}}\limits^{#1}_{#2}}}

\def\mapleft#1#2#3{\smash{\mathop{\hbox to
#3{\leftarrowfill}}\limits^{#1}_{#2}}}

\def\mapright#1#2{\smash{\mathop{\hbox to 0.90cm{\rightarrowfill}}\limits^{#1}_{#2}}}
\def\mapleft#1#2{\smash{\mathop{\hbox to 0.90cm{\leftarrowfill}}\limits^{#1}_{#2}}}

\def\mapleftright#1#2{\smash{\mathop{\hbox to 0.80cm{\leftarrowfill \rightarrowfill}}\limits^{#1}_{#2}}}

\title{Framed link presentations of 3-manifolds\\ by an $O(n^2)$ algorithm, II: colored complexes 
and boundings in their complexity
} 
\author{Sóstenes Lins and Ricardo Machado}

\date{\today}

\begin{document}

\maketitle

\begin{abstract}
This is part 2 of a 3-part article where we provide an $O(n^2)$-algorithm to produce a surgery 
presentation of a 3-manifold induced by a gem with a resolution. In this part we produce
a sequence of colored simplicial 2-complexes which are inverses and dual to the sequence of
gems produced in the first part. The refinements of the PL2-faces 
that keep appearing are idempotent: the second refinement of a PL2-face is isomorphic to its
first refinement. This fact inhibits exponentiability.
\end{abstract}

\section{Colored 2-complexes, 2-skeleton of $\mathcal{H}_m^\star$}

This is the second of 3 closely related articles.
References for the companion papers are \cite{linsmachadoA2012} and \cite{linsmachadoC2012}.

We start with $\mathcal{H}_1^\star$ which is easily obtained from $\mathcal{H}_1$.
We get $\mathcal{H}_m^\star$ from $\mathcal{H}_{m}$ and  $\mathcal{H}_{m-1}^\star$
by displaying the difference between $\mathcal{H}_m^\star$ and its predecessor.
The difference is precisely the balloon which becomes a pillow, encoded as
$(c:u$-$r)$, where:
$c\in\{0,1\}$ and $u, r$ are the odd vertices in the $J^2$-gem
defining the $c$-flip with $u,v=u+1$ corresponding to the 
2 PL3-faces of the balloon's head.
% These differences are exemplified 
%in Figs. \ref{fig:winglist01}-\ref{fig:winglist12}.
%The wings which also are shown are defined and used in the next chapter.

\subsection{The sequence of combinatorial 2-complexes: 
$\mathcal{H}_1^\star,$ $\mathcal{H}_2^\star,$ $\ldots,\mathcal{H}_{n}^\star$} 

For $m \in \{1,2,\ldots,n\}$ define 
$\mathcal{H}_m^\star$ to be the 2-skeleton of the dual of the gem  
$\mathcal{H}_m\backslash{2n}$. Removing vertex $2n$ corresponds to removing a 
tetrahedron from $S^3$ and so we know that $\mathcal{H}_m^\star$ embeds into $\mathbb{R}^3$.
Our goal is to obtain an explicit embedding of
$(\mathcal{J}^2)^\star=\mathcal{H}^\star_{n}$ into $\mathbb{R}^3$. 
The embedding will be 
{\em explicit} in the sense that
$\mathcal{H}_m^\star$ will be given by a 2-dimensional 
PL simplical complex where the $0$-simplices 
are endowed with $\mathbb{R}^3$-coordinates.

It is a simple matter to obtain the required embedding for the bloboid
$\mathcal{H}_1^\star$. Our initial strategy is to 
proceed from a combinatorial 2-complex
of $\mathcal{H}_m^\star$ to generate a combinatorial description of the next,
$\mathcal{H}_{m+1}^\star$. In the process we give upper bounds
for the number of $0$-, $1$- and $2$-simplices arising. These are quadratic polynomials
in $|V\mathcal{J}^2|$, the number of vertices of the gem $\mathcal{J}^2$. Note that
$|V\mathcal{J}^2|=|V\mathcal{G}|$, where $\mathcal{G}$ is the original resoluble gem.

The inverse of a 2-dipole thickening is a {\em 3-dipole slimming}. We need to consider the combinatorial 
duals of some objects. To a 2-dipole corresponds a \index{pillow} {\em pillow} in the dual, namely two tetrahedra
with two faces in common. To a 3-dipole (or blob) and a color involved in it corresponds
a {\em balloon} \index{balloon} in the dual. The balloon is formed by two tetrahedra with 3 faces in common
together with a triangle sharing an edge with the two tetrahedra. The dual of a 3-dipole
slimming is a \index{balloon-pillow move} {\em balloon-pillow move}, or {\em $bp$-move}. \index{bp-move} A sequence of these moves related 
to the $r^{24}_5$-example are depicted in Figs.  \ref{fig:Hstarsequence01} 
and \ref{fig:Hstarsequence02}.

There is a simple topological interpretation between primal and dual complexes, given in 
\cite{lins1995gca} pages 38, 39. Let's take a look 
at this interpretation in our context. This will help 
 to understand the $PL$-embedding $\mathcal{H}_m^\star$.
In what follows the $k$ in PL$k$-face means the dimension $k\in\{0, 1, 2,3\}$ of the PL-face.

\begin{itemize}
 \item [i.] 
a vertex $v$ in $G$ $\rightleftharpoons$ a 
solid PL-tetrahedron or PL3-face,
denoted by $\nabla_v$ in the dual of the gem whose
PL0-faces are labelled $z_0$, $z_1$, $z_2$ e $z_3^v$; 
in this work is enought to work with the 
boundary of a PL3-face; this is topologically a sphere
$S^2$ with four PL2-faces one of each color; the 3-simplices forming
a PL3-face need not be explicitly specified;
 \item [ii.] an $i$ colored edge $e_i$ in $G$ $\rightleftharpoons$ a set 
of $i$-colored 2-simplices defining 
a PL2$_i$-face in the dual of the gem;
 \item [iii.] a bigon $B_{ij}$ using the colors $i, j$ in 
$G$ $\rightleftharpoons$ a set of  1-simplices $b_{ij}$
 in $\mathcal{H}_n^\star$ defining a 
PL1$_{ij}$-face;
 \item [iv.] an $\widehat{i}$-residue $V_i$ in $G$ 
$\rightleftharpoons$ a 0-simplex in $\mathcal{H}_n^\star$ defining a PL0$_i$-face.
\end{itemize}

We define the combinatorial 2-dimensional PL complex $\mathcal{H}_1^{\star}$ as follows.
%Let $j \in \{1, \ldots, 2n\}$, define $v_0^j, v_1^j$ and $v_2^j$ as the vertices of $2n$ 
%equilateral triangles and subdivide as in 
 %defining $v_3^j, v_4^j$ and $v_5^j$. The scheme of colors
%depends of parity, if $j$ is odd, all the triangles are 3-colored, 
%otherwise triangles $v_0^jv_3^jv_1^j$ are 2-colored, defining the PL2$_2$-faces, 
%triangles $v_0^jv_4^jv_3^j$ and $v_2^jv_3^jv_4^j$ 
%are 1-colored, defining the 
%PL2$_1$-faces, and the triangles $v_1^jv_3^jv_5^j$ and $v_2^jv_5^jv_3^j$ are 0-colored, defining the PL2$_0$-faces.

%-----------------------------------
\begin{figure}[!htb]
\begin{center}
\includegraphics[scale=1]{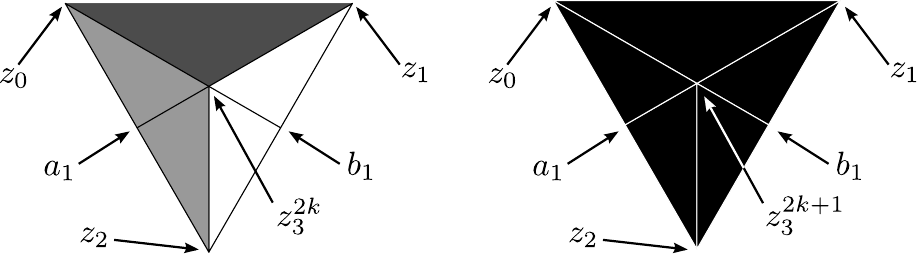}
\caption{PL2-faces of $\mathcal{H}^\star_1$.}
\label{fig:dualPL2}
\end{center}
\end{figure}
%-----------------------------------

The 0-simplices $z_0$, $z_1$ and $z_2$ are positioned in clockwise order 
as the vertices of an equilateral triangle 
of side $\mathcal{\varphi}$ in the 
$xy$-plane so that $z_0z_1$ is parallel to the $x$-axis and the center of the triangle
coincides with the the origin of an $\mathbb{R}^3$-cartesian system. 
The 0-simplex $a_1$ is $\frac{z_0+z_2}{2}$. The 0-simplex
$b_1$ is $\frac{z_2+z_1}{2}$.
Let the 0-simplices $z_3^j$ be defined as $z_3^j=(0,0,(2n-j)\phi)$, $1\le j \le 2n$, where $\phi$, as $\varphi$,
is a positive constant, see Fig. \ref{fig:dualPL2}. 
%We can easily adjust $\phi$ as a function of $\varphi$ so that
%the 3-simplex $z_0z_1z_2z_3^1$ becomes a regular tetrahedron. 
It is convenient, to leave $\phi$ and $\varphi$ as independent arbitrary positive constants, 
for adjusting the visual aspect of the embedded $\mathcal{H}_n^\star$.

Suppose $u$ is an odd vertex of the $J^2$-gem, $u'=u-1$, $v=u+1$ and $v'= v+1$.
The dual of a $\widehat{3}$-residue is $z_3^j$ where $j$ is even. 
When $j$ is odd, then $z_3^j$ is a $0$-simplex in the middle of a PL2$_3$-face,
incident to five 2-simplices of color 3.
The dual of the 03-gon is the PL1-face formed by the pair of 1-simplices $z_1b_1$ 
and $b_1z_2$. The dual of the 13-gon is the PL1-face formed by the pair of 
1-simplices $z_0a_1$ and $a_1z_2$. 
The dual of the 23-gon is the PL1-face formed by the 1-simplex $z_0z_1$.
The dual of the 01-gon relative to vertices $u$ and $v$ is the 1-simplex $z_2z_3^v$. 
The dual of the 02-gon relative to vertices $u$ and $v$ is the 1-simplex $z_1z_3^v$.
The dual of the 12-gon relative to vertices $u$ and $v$ 
is the 1-simplex $z_0z_3^v$.
The dual of a 3-colored edge $u'u$ is the image of PL2$_3$-face with odd index $u$ in the vertices.
The dual of an $i$-colored 
edge $uv$ with $i\in\{0, 1, 2\}$ is the PL2$_i$-face with even index $v$.

Before presenting $\mathcal{H}_m^\star$, $1\le m \le n,$ and its embeddings, we need to understand
the dual of the $(pb)^\star$-move and its inverse. In the primal, to apply a $(pb)^\star$-move, we need
a blob and a 0- or 1-colored edge. The dual of this pair is the {\em balloon:} the \index{balloon's head} {\em balloon's head} is
the dual of the blob; the {\em balloon's tail} \index{balloon's tail} is the dual of the $i$-edge. To make it easier to
understand, the $(pb)^\star$-move can be factorable into a 3-dipole move followed by a 2-dipole move, 
so in the dual, it is a smashing of the head of the balloon followed by the pillow move 
described in the book \cite{lins1995gca}, page 39.
This composite move is the {\em balloon-pillow move} \index{balloon-pillow move} or \index{bp-move} {\em bp-move}. Considering it as a single move 
is easier to implement and the code is faster. Restricting our basic change in the colored 2-complex 
to $bp$-moves we have nice theoretical properties which are responsible for avoiding an exponencial
proccess. In what follows we describe the $bp$-move assuming that the balloon's tail is 0-colored
using a generic balloon's tail, of which we just draw the contour.
The other case, color 1, is similar.

\begin{figure}[!htb]
\begin{center}
\includegraphics[width=14cm]{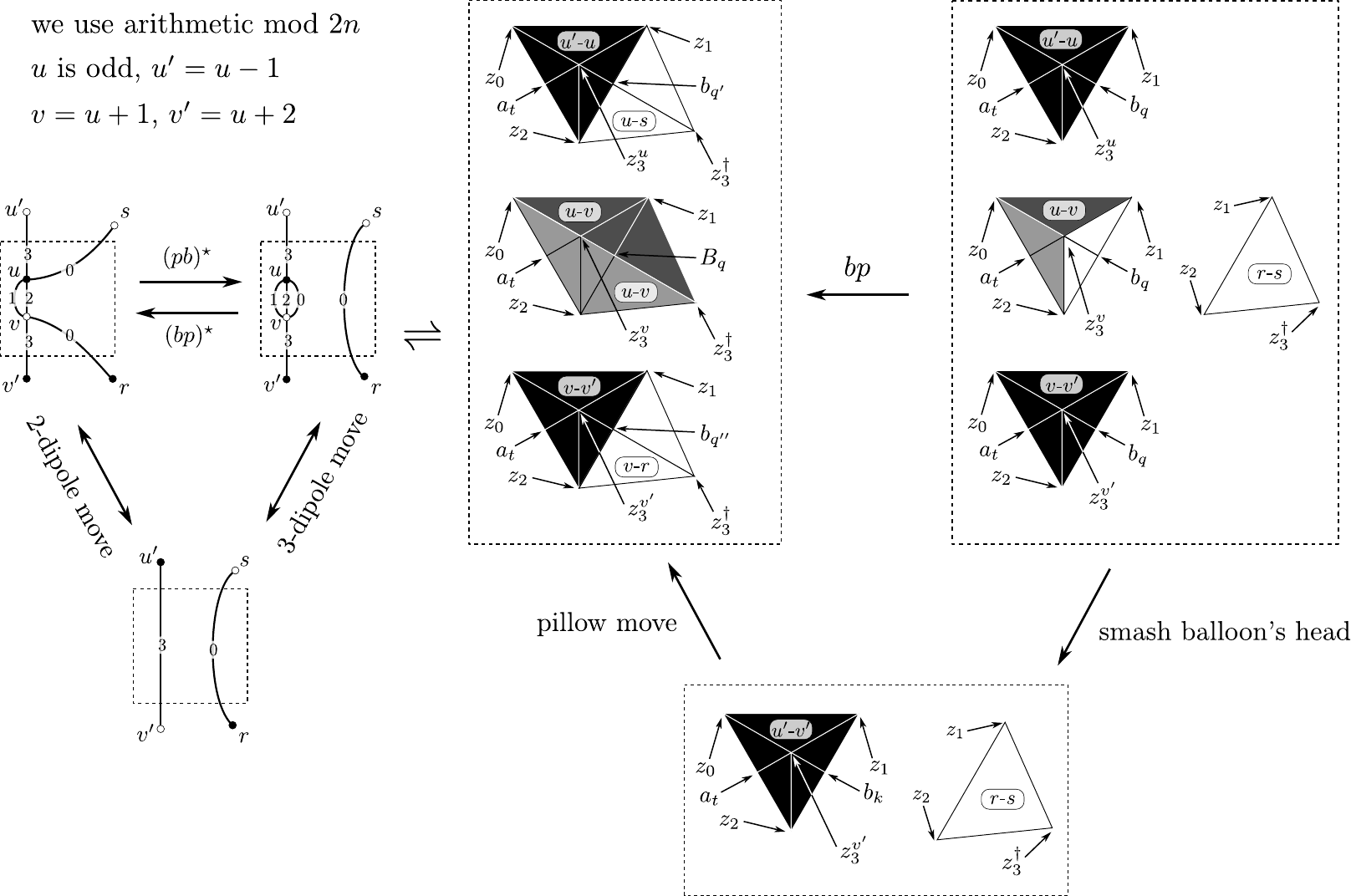}
\caption{Primal and dual $bp$-moves.}
\label{fig:pillowother2NOVO}
\end{center}
\end{figure}

\begin{itemize}
 \item [i.] if the image of $v_5^u$ and $v_5^v$ is $b_q$, create two 0-simplices $b_{q'}$ and $b_{q''}$, 
define the images of $v_5^u$ and $v_5^{v'}$ as $b_{q'}$ 
and $b_{q''}$ and change the label of the image of $v_5^{v}$ from $b_q$ to $B_q$;
 %(if balloon's tail is 1-colored, to be the image of $v_4^u$ and $v_4^{v'}$)
 \item [ii.] make two copies of the PL2$_0$-face, if necessary, 
refine each, from the middle vertex of the segment $z_2z_1$ to the third vertex $z_3^\dagger$, where $\dagger=j,$ for an adequate height $j$;
 \item [iii.] change the colors of the medial layer of the pillow as specified by the 
dual structure, namely by the current $J^2B$-gem, see Fig. \ref{fig:pillowother2NOVO}.
\end{itemize}

%This definition has some changes when the balloon's tail is 1-colored, the changes are in 1. the vertex are $v_4^u$ and $v_4^{v'}$, and the coloration, which the $2PL$-faces of $u$-$v$, here they are 0- and 2-colored.

The choice of the letters $P, B, R, G$ in the set {\em types} of next proposition 
comes from the colors $0=(P)ink$, $1=(B)lue$, $2=(R)ed$ and $3=(G)reen$. In the next proposition $R_{2k-1}^b$ is a PL2$_2$-face which is inside
the pillow neighboring a PL2$_1$-face. Similarly  $R_{2k-1}^p$ is a PL2$_2$-face which is inside
the pillow neighboring a PL2$_0$-face.
\begin{proposition}
\label{prop:allkinds}
Each PL2-face of the combinatorial simplicial complex $\mathcal{H}_m^\star$,
$1\le m \le n$, is isomorphic to 
one in the {\em set of types of triangulations} \index{types of triangulations}
$$\{G, P_{2k-1}, P_{2k-1}', B_{2k-1}, B_{2k-1}', R_{2k-1}^b, R_{2k-1}^p \ | 
\ k\in \mathbb{N} \},$$
described in Fig. 
\ref{fig:allkinds3}, where the index means the number of 
edges indicated and is called the {\em
rank of the type}. \index{rank of the type} Moreover, the PL2-faces that appear, as duals of the gem edges, 
have the minimum number of 2-simplices.
\end{proposition}

\begin{figure}[!htb]
\begin{center}
\includegraphics[width=15cm]{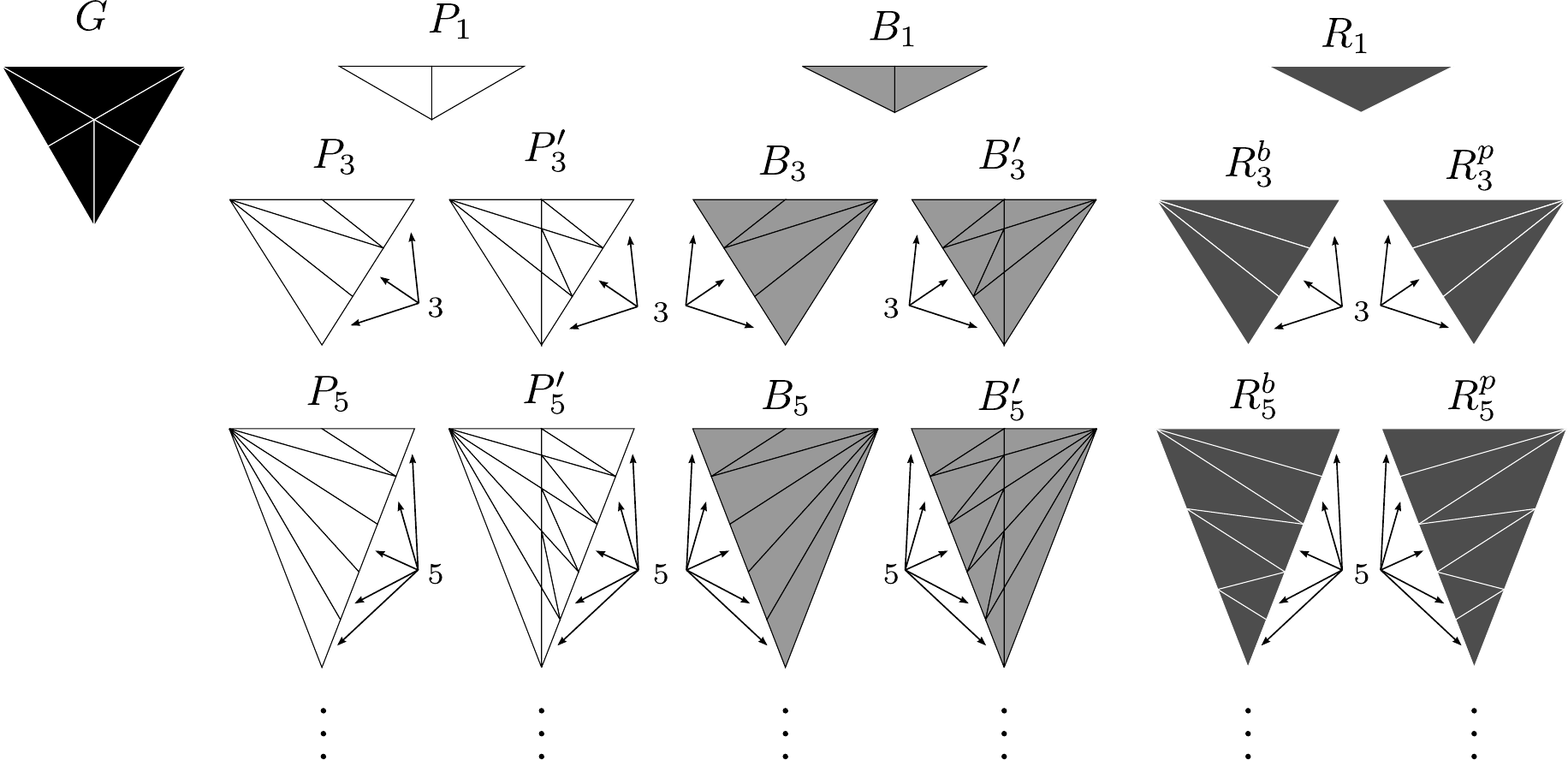}
\caption{All kinds of PL2-faces.}
\label{fig:allkinds3}
\end{center}
\end{figure}

\begin{proof}
We need to fix a notation for the head of the balloon, instead of drawing all the PL2-faces 
of the head, 
we just draw one PL2$_3$-face and put a label $u'$-$v'$. If the balloon's tail, is 
of type $P_1$, by applying a $bp$-move we can see at Fig. \ref{fig:pillowother5}
\begin{figure}[!htb]
\begin{center}
\includegraphics[width=12cm]{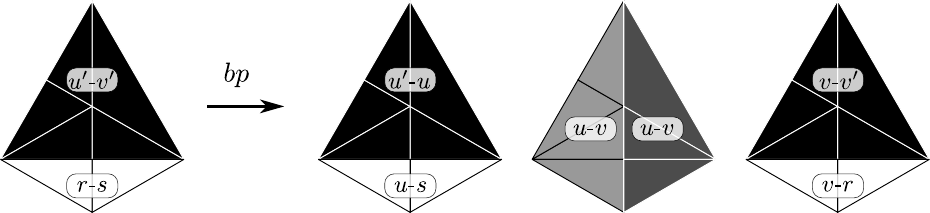}
\caption{Pillow-balloon move with balloon's tail of type $P_1$.}
\label{fig:pillowother5}
\end{center}
\end{figure}
that we get a PL2$_1$-face of type $B_3$ and a PL2$_2$-face of type $R_3^b$. 
The others PL2-faces are already known. 
If the balloon's tail, is of type $B_3$, by applying a $bp$-move, we need to refine the tail 
and the copies, otherwise we would not be able to build a pillow
 because some 2-simplices would be collapsed, so
we get two PL2$_1$-faces of type $B_3'$, one PL2$_0$-face of type $P_5$ and a PL2$_2$-face
 $R_5^p$.  The others PL2-faces are already known.

\begin{figure}[!htb]
\begin{center}
\includegraphics[width=12cm]{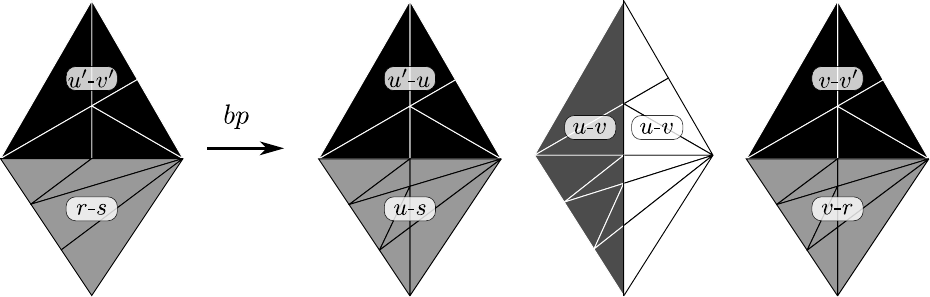}
\caption{Pillow-balloon move with balloon's tail of type $B_3$.}
\label{fig:pillowother3}
\end{center}
\end{figure}

In what follows given $X \in \{ P_{2k-1}', B_{2k-1}' \}$ denote by 
$\widehat{X}$ the copy of $X$ which is a PL2-face of the PL-tetrahedra 
whose PL2$_3$-face is below the similar PL2$_3$-face of the other PL-tetrahedra 
which completes the pillow in focus. 
In face of these conventions, if balloon's tail is of type
\begin{itemize}
 \item $P_{2k-1}$, then by applying a $bp$-move, we get types
$P_{2k-1}', \widehat{P}_{2k-1}', B_{2k+1}, R_{2k-1}^b $%\backslash P_{2k-1}$
 \item $P_{2k-1}'$, then by applying a $bp$-move, we get types 
$P_{2k-1}, B_{2k+1}, R_{2k-1}^b$ 
 \item $B_{2k-1}$, then by applying a $bp$-move, we get types
$B_{2k-1}',  \widehat{B}_{2k-1}', P_{2k+1}, R_{2k-1}^p$% \backslash B_{2k-1}$
 \item $B_{2k-1}'$, then by applying a $bp$-move, we get types 
$B_{2k-1}, P_{2k+1}, R_{2k-1}^p$ 
\end{itemize}
\end{proof}

 In the sequel we will see that with this combinatorics
attached to the PL2-faces the combinatorial $\mathcal{H}_m^\star$'s can be PL-embedded into $\mathbb{R}^3$.
It is whorthwhile to mention, in view of the above proof, that each PL2-face is refined at most one time. 
So, if $X$ is a type of PL2-face, $X'$ is its refinement, then
$X''=X'$. This idempotency is a crucial property inhibiting 
the exponentiality of our algorithm.

\subsection{Upper bounds for the number of simplices of the complex $\mathcal{H}_{n}^\star$}

Now we give quadratic upper bounds for the number of $i$-simplices, 
$i\in \{0, 1, 2\}$ of $\mathcal{H}_n^\star$.

\begin{lemma}
\label{theo:thequadratic}
The quadratic expressions $$3n^2-5n+9,~~~ 11n^2-17n+21,~~~ 8n^2-10n+12$$
 are upper bounds for the numbers of 0-simplices, 1-simplices and 2-simplices
of the colored 2-complex $\mathcal{H}_{n}^\star$ induced by a resoluble
gem $G$ with 2n vertices.
\end{lemma}
\begin{proof} 

\subsection*{case $i = 0$}

We know that $\mathcal{H}_{1}^\star$ has exactly 
$z_0, z_1, z_2, a_1, b_1$ and $z_3^j, j\in \{1,\ldots,2n\}$ as 0-simplices, which is $2n+5$
 0-simplices.
In first step balloon's tail has to be of type $P_1$ or $B_1$, so by applying $bp$-moves, 
we get two new 0-simplices, which inverse image is a black and 
white disk in Fig. \ref{fig:contagem1} (first part)

In second step the worse case is when balloon's 
tail is of type $P_3$ or $B_3$, generated by last $bp$-move, 
so we add $6\times1+2=8$ to the number of 0-simplices in the upper bound. See Fig. \ref{fig:contagem1} (second part).
\begin{figure}[!htb]
\begin{center}
\includegraphics[width=12cm]{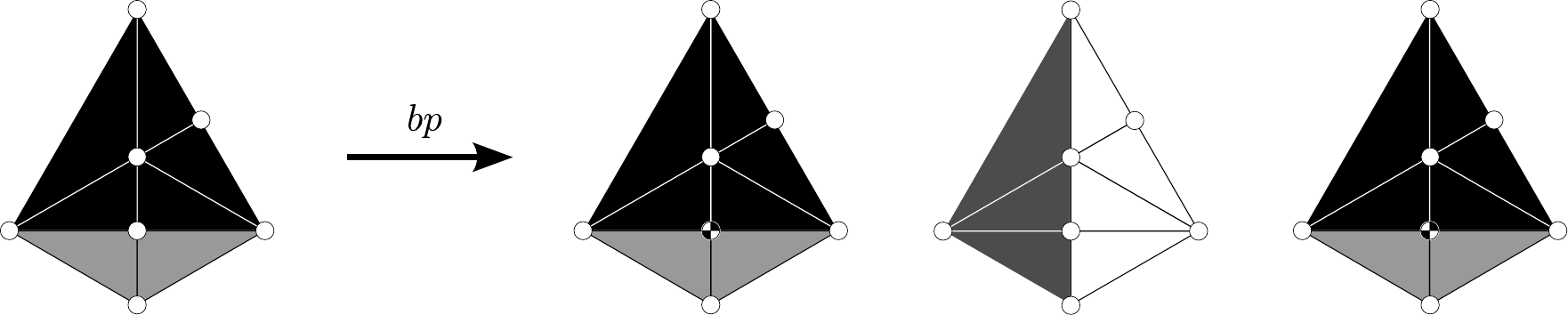}\\
\vspace{7 mm}
\includegraphics[width=12cm]{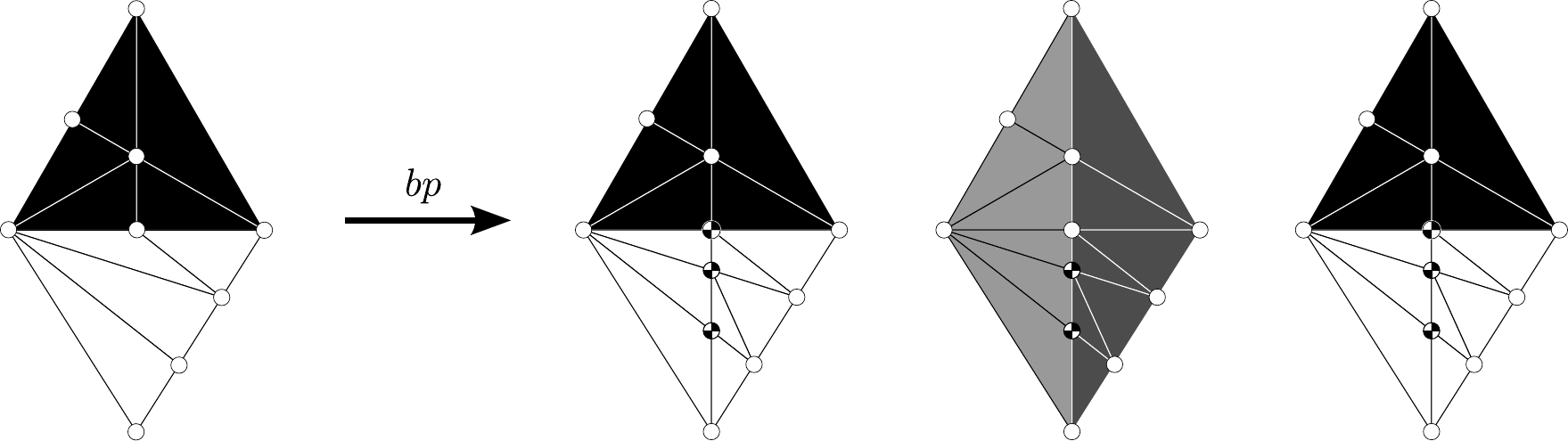}
\caption{Upper bound for the number of the 0-simplices, first and second steps.}
\label{fig:contagem1}
\end{center}
\end{figure}

In step $k$ we note that worst case is when we use the greatest ranked PL2-face generated by 
last $bp$-move, so it means that balloon's tail has to be of type $P_{2k-1}$ or $B_{2k-1}$, 
(in Fig. \ref{fig:contagem3} $j=2k-1$)
which we add $6\cdot(k-1)+2$ 0-simplices. 
\begin{figure}[!htb]
\begin{center}
\includegraphics[width=12cm]{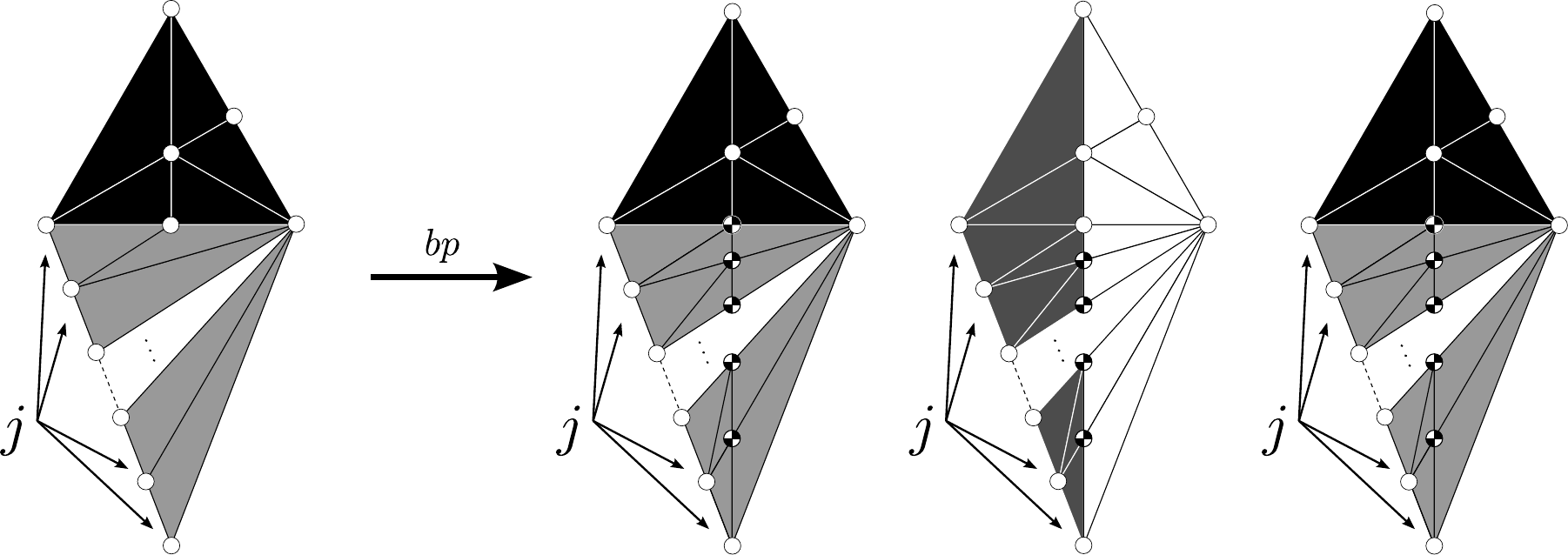}
\caption{Upper bound for the number of the 0-simplices, $j$-th step.}
\label{fig:contagem3}
\end{center}
\end{figure}
By adding the number of 0-simplices created by $bp$-moves from step 1 until step $k$ 
we get  $3k^2-k$ 0-simplices. As the number of steps is $n-1$, 
and we have at the begining $2n+5$, we have that $3n^2-5n+9$ is 
an upper bound for the number of 0-simplices.

\subsection*{case $i = 1$}

We know that $\mathcal{H}_{1}^\star$ has exactly $z_0z_1, z_0a_1, z_1b_1, z_2b_1, z_2a_1$,  
$z_0z_3^j,$ $z_1z_3^j,$ $b_1z_3^j,$ $z_2z_3^j,$ $a_1z_3^j,$ $j\in \{1,\ldots,2n\}$ as 1-simplices
which give us  $10n+5$ 1-simplices.
In first step balloon's tail has to be of type $P_1$ or $B_1$, so by applying $bp$-moves, 
we add $2\times3=6$ to the number of 1-simplices in the upper bound.

In second step we know that worst case is when balloon's tail is of type $P_3$ or $B_3$, see Fig. 
\ref{fig:cont1simplex1} (second part) and 
we add $8+2\times11 -2=28$ to the number of 1-simplices in the upper bound.
\begin{figure}[!htb]
\begin{center}
\includegraphics[width=12cm]{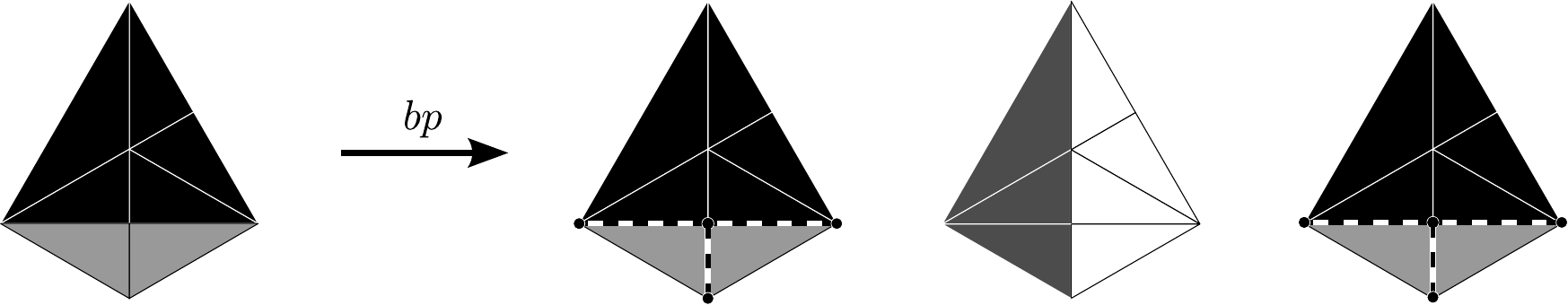}\\
\vspace{7 mm}
\includegraphics[width=12cm]{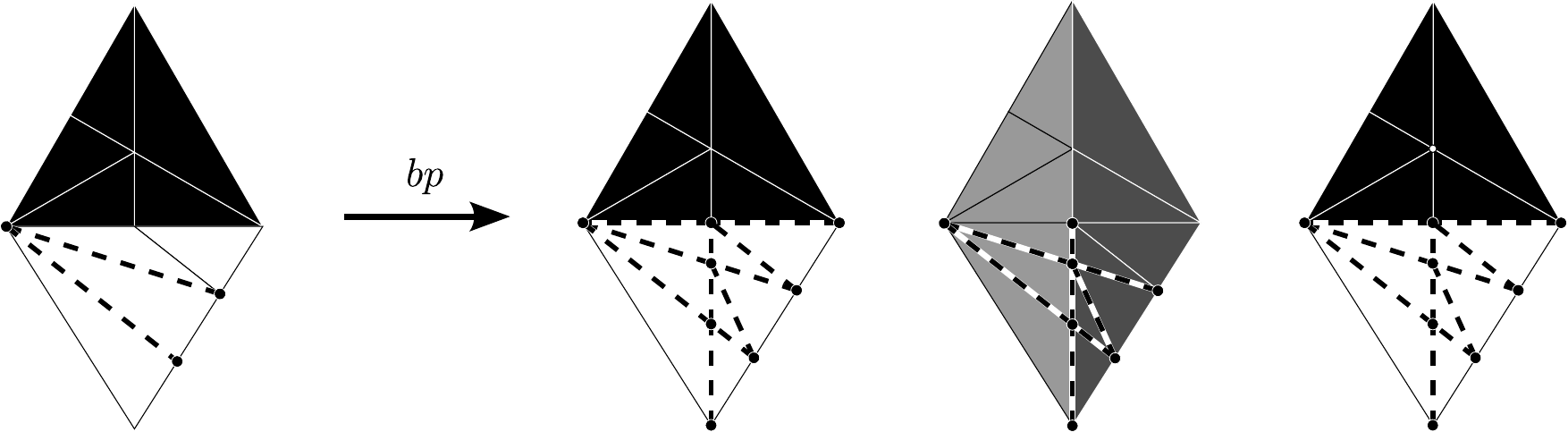}
\caption{Upper bound for the number of the 1-simplices, first and second steps.}
\label{fig:cont1simplex1}
\end{center}
\end{figure}

In step $k$,  as in the 0-simplex case, the worst case is when we use the greatest ranked PL2-face generated by
 last $bp$-move, so it means that balloon's tail has to be of type $P_{2k-1}$ or $B_{2k-1}$, 
(in Fig. \ref{fig:cont1simplex3}, $j=2k-1$) and we add $2(3+8(k-1))+8(k-1)-2(k-1)=22k-16$ to the 
number of 1-simplices in the upper bound.
\begin{figure}[!htb]
\begin{center}
\includegraphics[width=12cm]{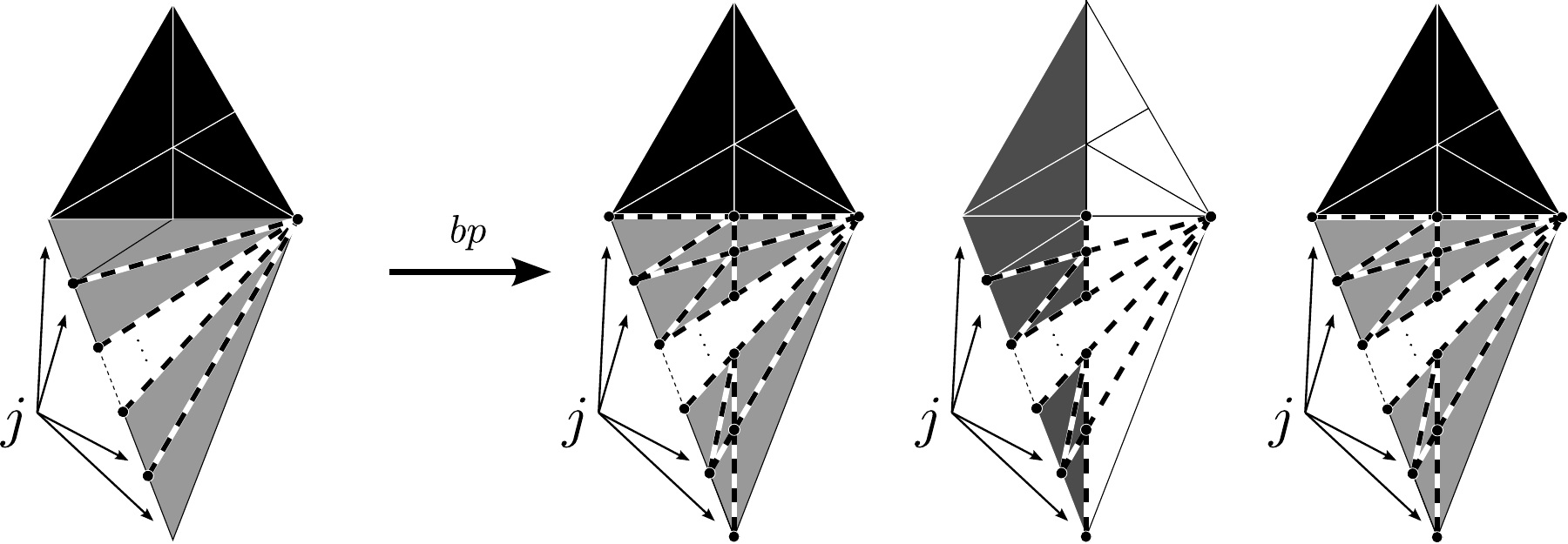}
\caption{Upper bound for the number of the 1-simplices, $j$-th step.}
\label{fig:cont1simplex3}
\end{center}
\end{figure}
An upper bound for the number of 1-simplices created by $bp$-moves from step 1 until step $k$ is
 $11k^2-5k$. As the number of steps is $n-1$, 
and we have at the begining $10n+5$, it follows that, 
by adding the arithmetical progression, $11n^2-17n+21$ is 
an upper bound for the number of 1-simplices. 

\subsection*{case $i = 2$}

We know that $\mathcal{H}_{1}^\star$ has exactly $z_0z_3^jz_1, z_1z_3b_1, z_2b_1z_3^j, z_2z_3^ja_1, z_0a_1z_3^j$,
 $j\in \{1,\ldots,2n\}$ as 2-simplices
which give us  $10n$ 2-simplices.

In first step balloon's tail has to be of type $P_1$ or $B_1$, so by applying $bp$-moves, we add $2\times2$ 2-simplices.
In second step we add $3\times8$ and subtract $4$

\begin{figure}[!htb]
\begin{center}
\includegraphics[width=12cm]{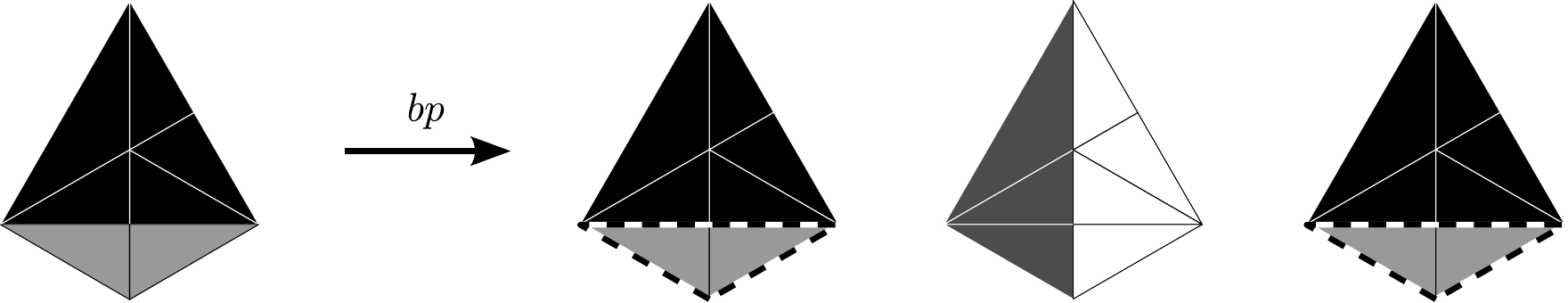}\\
\vspace{7 mm}
\includegraphics[width=12cm]{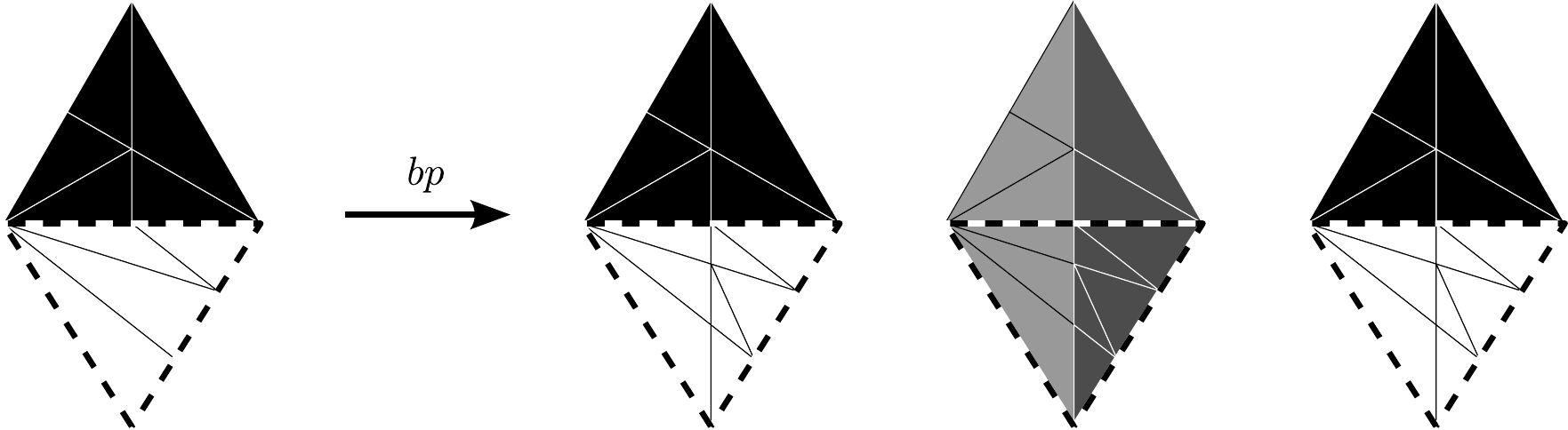}
\caption{Upper bound for the number of the 2-simplices, first and second steps.}
\label{fig:cont2simplex1}
\end{center}
\end{figure}

\begin{figure}[!htb]
\begin{center}
\includegraphics[width=12cm]{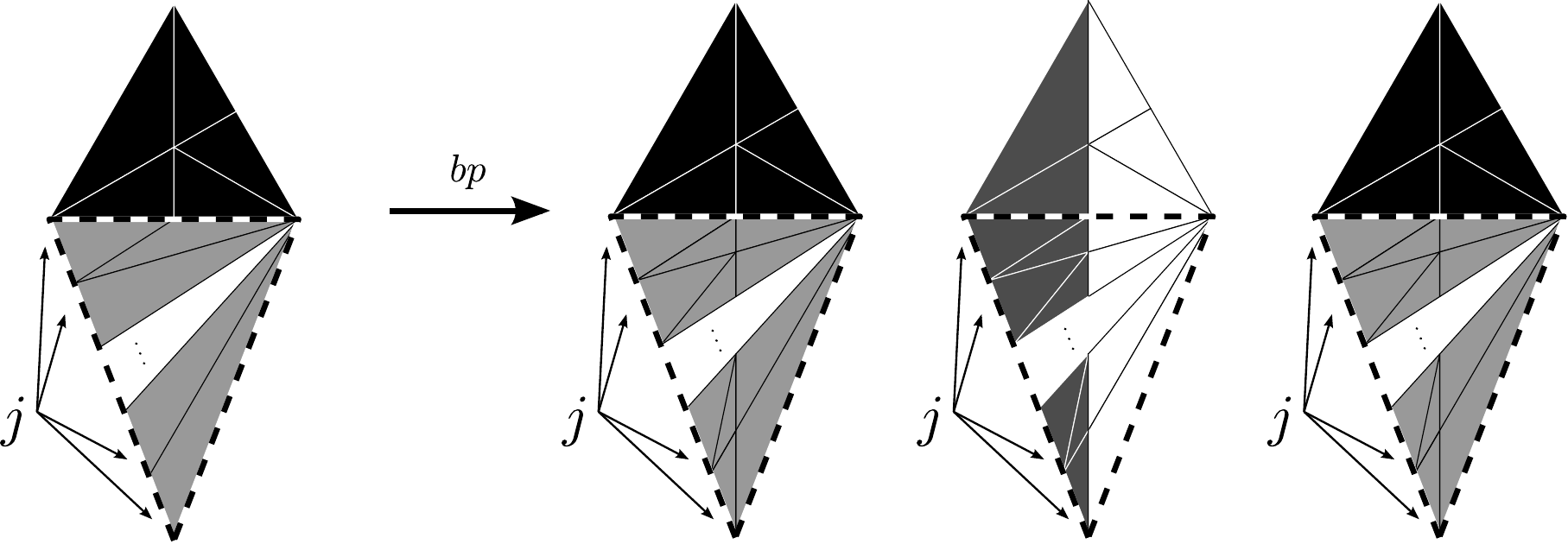}
\caption{Upper bound for the number of the 2-simplices, $j$-th step.}
\label{fig:cont2simplex4}
\end{center}
\end{figure}

As we know, the worst case in step $k$ is when balloon's tail is of type $P_{2k-1}$, or $B_{2k-1}$. 
By apply $bp$-move we add $3\cdot(6\cdot(k-4))$ and subtract $2k$, for $k\geq2.$

By adding the number of 2-simplices created by $bp$-moves from step 1 until step $k$ 
we get $8k^2-4k$ 2-simplices. As the number of steps is $n-1$, 
and we have at the beginning $10n$ 2-simplices, $8n^2-10n+12$ is 
an upper bound for the number of 2-simplices. 
\end{proof}

%-----------------------------------
\bibliographystyle{plain}
%\bibliographystyle{is-alpha}
%\addcontentsline{toc}{bibliografia}{\MakeTextUppercase{Referências Bibliográficas}}
%\bibliography{d:/slsl\3.DadosSostenes.35.ArtigosLivros.bibtexGoogleScholar/bibtexIndex.bib} % bib file is slsl.bib
%\bibliography{~/home/ricardo/Dropbox/35.ArtigosLivros.bibtexGoogleScholar/bibtexIndex.bib}
\bibliography{bibtexIndex.bib}
%\bibliography{slsl}

% \printindex

\subsection{Appendix A:\\
 sequence of $bp$-move corresponding to $r_5^{24}$}

%-----------------------------------
\begin{figure}
\begin{center}
\includegraphics[width=10.5cm]{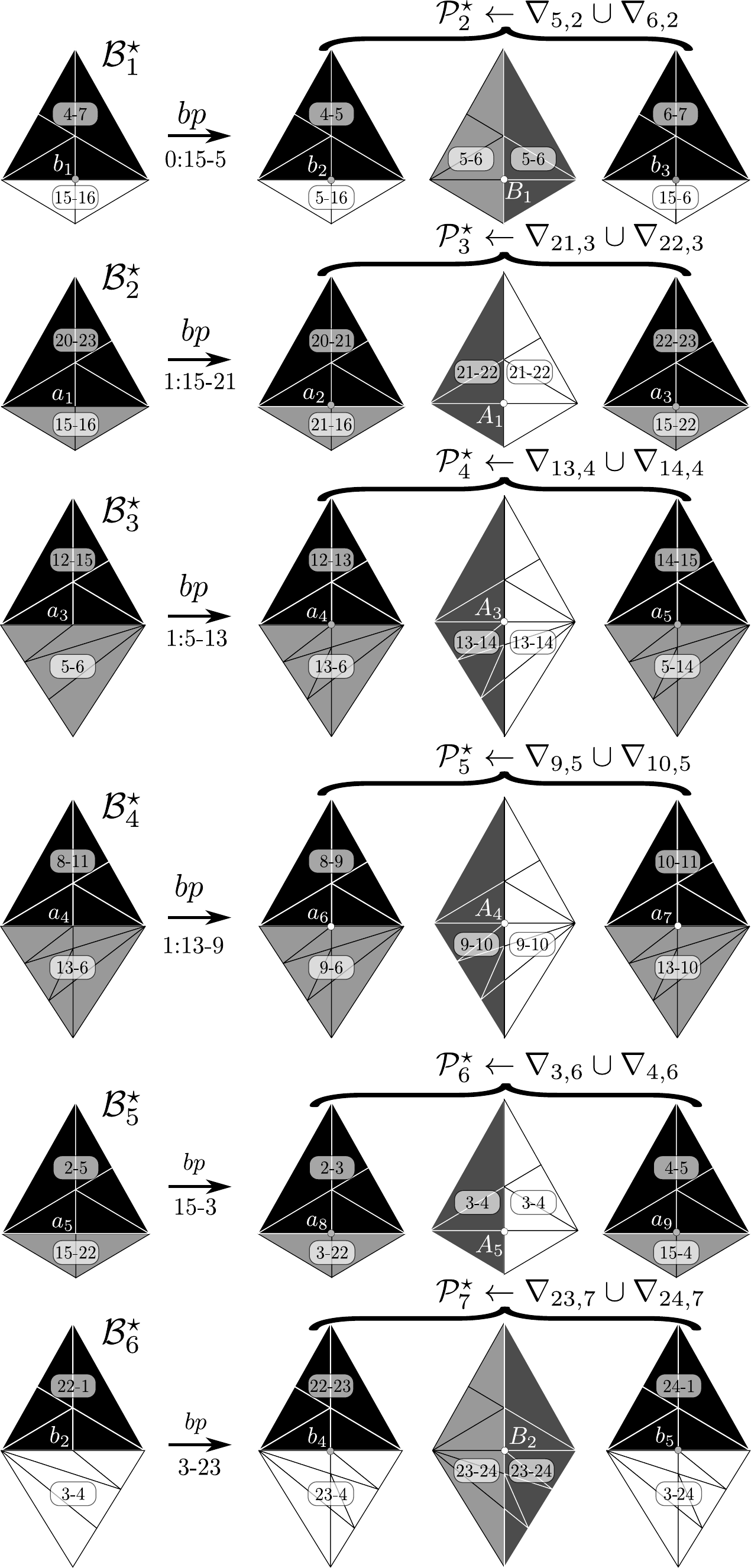} \\
\caption{\sf sequence of $bp$-moves, $m = 1,\ldots,6$  ($r_5^{24}$-example).}
\label{fig:Hstarsequence01}
\end{center}
\end{figure}
%-----------------------------------

%-----------------------------------
\begin{figure}
\begin{center}
\includegraphics[width=10.5cm]{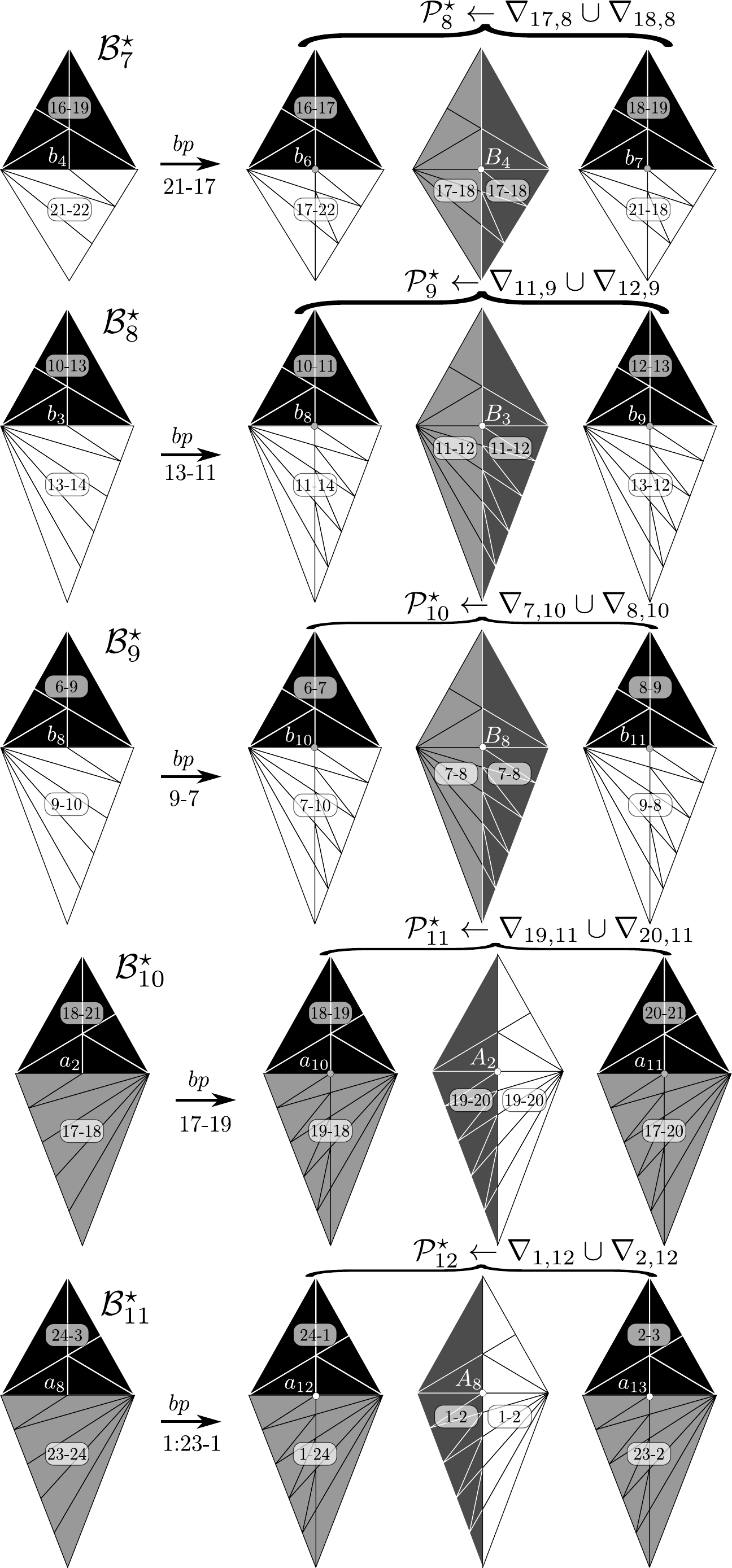} \\
\caption{\sf sequence of $bp$-moves, $m = 7,\ldots,11$  ($r_5^{24}$-example).}
\label{fig:Hstarsequence02}
\end{center}
\end{figure}
%-----------------------------------

\vspace{10mm}
\begin{center}

\hspace{7mm}
\begin{tabular}{l}
   S\'ostenes L. Lins\\
   Centro de Inform\'atica, UFPE \\
   Recife--PE \\
   Brazil\\
   sostenes@cin.ufpe.br
\end{tabular}
\hspace{20mm}
\begin{tabular}{l}
   Ricardo N. Machado\\
   Núcleo de Formação de Docentes, UFPE\\
   Caruaru--PE \\
   Brazil\\
   ricardonmachado@gmail.com
\end{tabular}

\end{center}

\end{document}